\documentclass{amsart}

\usepackage{amsmath}
\usepackage{amssymb}
\usepackage{stmaryrd}

\usepackage{etex}

\usepackage{color}
\usepackage{pictexwd,dcpic}
\usepackage{mathtools}
\usepackage{upref}
\usepackage{xy}
 \usepackage{graphicx}
\usepackage[mathscr]{eucal}
\input xy
\xyoption{all}
\usepackage{hyperref}
\usepackage{url}
\hypersetup{colorlinks,%
citecolor=blue,%
filecolor=black,%
linkcolor=black,%
urlcolor=blue,%
}

\makeatletter
\def\section{\@startsection{section}{1}%
  \z@{.7\linespacing\@plus\linespacing}{.5\linespacing}%
  {\normalfont\bfseries\centering}}
\def\@secnumfont{\bfseries}
\makeatother

\def\sideremark#1{\ifvmode\leavevmode\fi\vadjust{\vbox to0pt{\vss%
  \hbox to 0pt{\hskip\hsize\hskip1em%
  \vbox{\hsize3cm\tiny\raggedright\pretolerance10000%
  \noindent #1\hfill}\hss}\vbox to8pt{\vfil}\vss}}}%

\let\ccdot\cdot
\def\cdot{\hbox to 2.5pt{\hss$\ccdot$\hss}}

\newcommand{\ga}{\gamma}

\renewcommand{\phi}{\varphi}

\def\>{\rightarrow}

\def\EE{\Bbb E}
\def\MM{\Bbb M}
\def\RR{\Bbb R}
\def\ZZ{\Bbb Z}

\newtheoremstyle{note}
{6pt}
{4pt}
{}
{}
{\itshape}
{:}
{.5em}
{}

\newtheorem*{prop*}{Proposition}
\newtheorem*{thm*}{Theorem}
\newtheorem*{lem*}{Lemma}
\newtheorem*{cor*}{Corollary}

\newtheorem{thm}{Theorem}[section]

\newenvironment{remark}[1][Remark.]{\begin{trivlist}
\item[\hskip \labelsep {\bfseries #1}]}{\end{trivlist}}

\numberwithin{equation}{section}

\begin{document}
\title[The conformal pseudodistance]{The conformal pseudodistance\\ and null geodesic incompleteness\\ 
}
\author{Michael J. Markowitz}
\address{Information Security Corporation, 1011 Lake Street, Suite 425, Oak Park, IL 60301}
\email{markowitz@infoseccorp.com}
\date{\today}

\begin{abstract}

In this note we clarify the relationship between the null geodesic completeness of 
an Einstein Lorentz manifold and its conformal Kobayashi pseudodistance.
We show that an Einstein manifold has at
least one incomplete null geodesic if its pseudodistancfe is nontrivial. If its pseudodistance is
nondegenerate, all of its null geodesics must be incomplete. Thus an Einstein manifold $(M,g)$ 
has no complete null geodesic if there is a ``physical metric'' in the conformal class of $g$ 
satisfying the null convergence and null generic conditions.

\end{abstract}

\subjclass[2000]{53C50, 53A30, 53B30, 53C25, 53C22, 53C05, 53C80}
\maketitle

\setcounter{tocdepth}{2}

\medskip
\noindent{\bf Acknowledgments.} The perspective described herein was obtained during the first week of the ``Workshop on Cartan Connections, Geometry of Homogeneous Spaces, and Dynamics'' held at the International Erwin Schr\"odinger Institute (ESI) in Vienna (July 11-15, 2011). The author would like to thank the organizers, Andreas \v Cap, Charles Frances, and Karin Melnick, for inviting him to participate in the conference and providing a very stimulating environment. We also thank Charles, Thierry Barbot, Ben McKay, and Roger Schlafly for several interesting discussions on this and related topics.


\section{Introduction}\label{1}

In \cite{M1} we introduced on each Lorentz manifold $(M,g)$ an intrinsic pseudodistance $d_M^{[g]}$ depending only 
on the underlying conformal class $[g]$ of $g$. $d_M^{[g]}$ is called the \emph{conformal Kobayashi pseudodistance}
because its construction is modeled on that of the projective Kobayashi pseudodistance introduced in \cite{K1}.
For an extension of this construction to general parabolic geometries, see \cite{M3}.

The purpose of the present note is to make precise the relationship 
between the behavior of $d_M^{[g]}$ and the (affine) incompleteness of null geodesics when $g$ 
is an Einstein metric. In particular, we note the potential usefulness of $d_M^{[g]}$ in the study of 
spacetime singularities characterized by null geodesic incompleteness of an Einstein scale.

The following will be shown to be a simple consequence of the definition of $d_M^{[g]}$.

\begin{prop*}
If $(M,g)$ is a connected Einstein manifold and $d_M^{[g]}$ is nontrivial, then $g$ has at least one incomplete null geodesic.
If $d_M^{[g]}$ nondegenerate, then every null geodesic of $g$ is incomplete.
\end{prop*}

We now recall the definitions of two positive energy conditions used in \cite{M1}:\footnote
{For a detailed discussion of these conditions, also see \cite{HE}, p. 95 and 101.}
the \emph{null convergence condition},
\begin{align*}
& Ric(X,X) \geq 0 \quad \text{for all null vectors } X, \tag{NCC}
\end{align*}
and the \emph{null generic condition},
\begin{align*}
& Ric(\gamma',\gamma') \neq 0 \text{ at some point along every inextendible null geodesic } \gamma. \tag{NGC}
\end{align*}

Our main result, a consequence of the above proposition and the results of \cite{M1}, may now be stated as follows:

\begin{thm*}
Suppose that $(M,g)$ is an Einstein manifold and that some metric $\tilde g$ in the conformal class of 
$g$ satisfies the NCC and NGC. Then $d_M^{[g]}$ is nondegenerate and every null geodesic of $g$ is 
incomplete. If $M$ is compact, $\tilde g$ must be null geodesically incomplete.
\end{thm*}

\noindent
Turning this around, we have the equivalent formulation:

\begin{thm*}
Suppose that $(M,\tilde g)$ is a Lorentz manifold satisfying the NCC and NGC. 
Then $d_M^{[\tilde{g}]}$ is nondegenerate and any Einstein metric 
in the conformal class of $\tilde g$ has no complete null geodesic. Moreover, if $M$ is compact, 
$\tilde{g}$ must be null geodesically incomplete.
\end{thm*}

In section 2, we recall the definition of $d_M^{[g]}$ and review its most important properties.
An important example illustrating the theorem is discussed in section 3. 
The proofs of these results appear in section 4.

\medskip
\begin{remark}[Remarks.]
Given a conformally Einstein manifold for which $d_M^{[g]}$ is nondegenerate, we can only
show that extreme null geodesic incompleteness manifests itself in every Einstein scale.
As far as we know, the `physical metric' ($\tilde g$ in both statements of the theorem) may admit 
a complete null geodesic, though this is certainly not the case for Einstein--de Sitter space, 
an important example discussed below. In this regard, we should also mention the theorem
of Beem (\cite{BE}) to the effect that any distinguishing, stongly causal, stably causal, or
globally hyperbolic Lorentz manifold has a null (and timelike) geodesically complete metric
in its conformal class.

A conformally Einstein metric can be characterized as a normal conformal
parabolic geometry whose standard tractor bundle admits a (suitably generic) parallel 
section (see \cite{GN}).
According to \cite{KM}, 
a null geodesically complete Einstein metric (of dimension $n\geq 3$) is the only Einstein metric in 
its conformal class (up to constant rescalings). 
In light of the above theorem and the present results, it would be interesting to know if an Einstein metric 
for which $d_M^{[g]}$ is nondegenerate can admit a nontrivial conformal Einstein rescaling.
\end{remark}

\medskip
\section{The conformal Kobayashi pseudodistance}\label{2}

A \emph{pseudodistance} (or `\emph{pseudometric}') on a set $M$ is a function $d: M \times M \to [0,\infty)$ satisfying:
\begin{align*}
d(x,x) &= 0\\
d(x,y) &= d(y,x), \quad \text{and}\\
d(x,y) &\leq d(x,z) + d(z,y),
\end{align*}
for all $x,y,z \in M$.
$d$ is said to be \emph{nondegenerate} if it is a true distance: $d(x,y) = 0$ implies $x=y$, for all $x,y \in M$.

We now show how to define a conformally invariant pseudodistance
on a connected pseudo-Riemannian manifold $(M,g)$ (of indefnite but otherwise arbitrary signature) of
dimension $n\geq2$, using its class of (projectively parametrized) null geodesics. 
The role of a `standard measuring rod' in our construction will be played by the interval
$I = (-1,1) \subset \RR$ endowed with its \emph{Poincar\'e metric}
\begin{align*}
ds_I^2 &= {4du^2 \over (1-u^2)^2}
\end{align*}
and associated distance function
\begin{align*}
\rho_I(u_1,u_2) = \left|\,log\, {(1+u_1)(1-u_2) \over {(1-u_1)(1+u_2)}}\,\right|.
\end{align*}
The real projective analog of the Schwarz lemma says that general projective
transformations of $I$ are nonexpansive with respect to $\rho$, while those
which are isometries at a single point must be automorphisms; see \cite{K1}.

First recall that a \emph{projective parameter} along an affinely parametrized null geodesic 
$\gamma(s)$ of $g$ may be defined as a solution of the differential equation
\begin{align}\label{eq:sde}
\{p,s\} \,&=\, - \frac{2}{n-2}\, Ric\big(\gamma'(s),\gamma'(s)\big),
\end{align}
where primes denote derivatives with repect to the affine parameter $s$, 
$Ric$ is the Ricci tensor of $g$, and $\{p,s\}$ is the \emph{Schwarzian derivative}
\begin{align*}
\{p,s\} \,&:=\, \frac{p'''}{p'} - \frac{3}{2} \bigg(\frac{p''}{p'}\bigg)^2.
\end{align*}
Projective parameters are well-defined up to linear fractional transformations, \emph{i.e.},
reparametrizations of the form
$t \mapsto (at+b)/(ct+d)$ with $ad - bc \not= 0$. It is well-known that the class of null 
geodesics of $(M,g)$ together with these distinguished parameters is invariant under
conformal rescalings of $g$. (See \cite{M1} for details.)

Now, given $x,y \in M$, we define a \emph{Kobayashi path} from $x$ to $y$ 
to be a collection of points $x\!=\!x_0, x_1, \ldots, x_k\!=\!y \in M$, pairs of points
$a_1, b_1, \ldots, a_k, b_k \in I$, and projectively parametrized null geodesics,
{$\ga_1,\ldots,\ga_k\colon\! I \to M$} 
such that
\begin{align*}
\ga_i(a_i) &= x_{i-1} \text{ and } \ga_i(b_i) = x_i \text{ for } i = 1,\ldots,k.
\end{align*}

\noindent
Denoting such a path by $\alpha = \{a_i, b_i, \ga_i\}_{i=1}^k$, we define its  
\emph{length} to be
\begin{align*}
L(\alpha) = \sum_{i=1}^k \rho_I(a_i,b_i).
\end{align*}

\begin{figure}[h]
\includegraphics[width=4.8in, clip=true, trim=0 470 0 0]{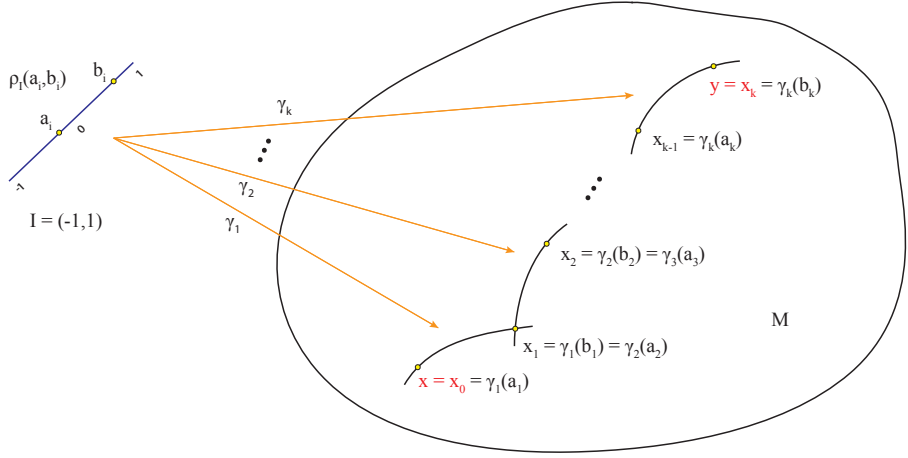}
\caption{A Kobayashi path from $x$ to $y$.}
\end{figure}
\vspace{20pt}

Finally, we define the \emph{conformal pseudodistance} by
\begin{align*}
d^{[g]}_M(x,y) = \inf_\alpha \,L(\alpha),
\end{align*}
where the infimum is taken over all Kobayashi paths $\alpha$ joining $x$ to $y$ in $M$. 
(Note that any two points of a connected Lorentz manifold $(M,g)$ may be joined by
a piecewise smooth null geodesic $\gamma$ and that such a $\gamma$ may easily be turned
into a Kobayashi path. Hence $d_M(x,y)$ is well-defined and 
non-negative on $M$.)

The following theorem summarizes the basic properties of $d^{[g]}_M$.

\begin{thm}[Markowitz \cite{M1}]\label{thm:1}
Suppose that $(M,g)$ is a connected Lorentz manifold.

(a) $d^{[g]}_M$ is a pseudodistance on $M$ that depends only on the conformal class $[g]$ of $g$.

(b) If $\gamma \colon\! I \to M$ is a projectively parametrized null geodesic, then 
\[
\rho_I(p,q) \geq d^{[g]}_M(\gamma(p),\gamma(q)) \quad \text{ for all } p,q \in I.
\]

(c) $d^{[g]}_M$ is the largest pseudodistance on $M$ with respect to property (b): if $\delta_M$ is any
pseudodistance on $M$ such that $\rho_I(p,q) \geq \delta_M(\gamma(p),\gamma(q))$ for all 
projectively parametrized null geodesics $\gamma \colon\! I \to M$ and all $p,q \in I$, then 
\[
\delta_M(x,y) \leq d^{[g]}_M(x,y) \quad \text{ for all } x,y \in M.
\]

(d) Each conformal automorphism $\phi \colon\! M \to M$ of $(M,g)$ is an isometry
with respect to $d^{[g]}_M$:
\[
d^{[g]}_M(\phi(x),\phi(y)) = d^{[g]}_M(x,y) \quad \text{ for all } x,y \in M.
\]

(e) If $(M,g)$ is a null geodescially complete Lorentz manifold and $Ric(X,X) \leq 0$ for all null vectors
$X$, then $d^{[g]}_M \equiv 0$.

(f) If $(M,g)$ is a Lorentz manifold satisfying the NCC and NGC, then $d^{[g]}_M$ is nondegenerate.
\end{thm}


\medskip
\section{An illustrative example}\label{3}

In this section we illustrate our theorem using the simple cosmological model
introduced by Einstein and de Sitter in 1932 (and possibly earlier by Friedman), 
referring the reader to \cite{SW} for all necessary 
background information. 

\emph{Einstein--de Sitter space}, $(\EE,\tilde{g})$, is defined to be
$\EE = \RR^3 \times (0,\infty)$ with the Lorentzian warped product metric $\tilde{g} = t^{4/3} \sum_{i=1}^3 dx^i \otimes dx^i -dt \otimes dt $.
$\EE$ is connected, oriented by $dx^1 \wedge dx^2 \wedge dx^3 \wedge dt$,
and time oriented by $\partial_t$. 
The isometry group of $(\EE,\tilde{g})$ may be identified with the Euclidean group $E(3) = O(3) \ltimes \RR^3$
acting in the obvious way on the first three coordinates.
The Ricci tensor of $\tilde{g}$ is given by
\begin{align}\label{Ric}
Ric = (2/3)t^{-2/3} \sum_{i=1}^3  dx^i \otimes dx^i + (2/3) t^{-2} dt \otimes dt,
\end{align}
and its scalar curvature by $R = {{4}\over{3}}t^{-2}$. From \eqref{Ric} we see
that $Ric(X,X) > 0$ for all null vectors $X$, so that $(\EE,\tilde{g})$ 
satisfies the NCC and the NCG. The nondegenerate conformal Kobayashi distance 
$d_\EE^{[\tilde{g}]}$ was computed explicitly in \cite{M2}.

Clearly $\tilde{g}$ is not an Einstein metric on $\EE$. In fact, $\tilde{g}$ satisfies the Einstein field equations
with the stress-energy tensor of a dust: $\hat{T} = R \,\partial_t \otimes \partial_t$.
Since $R \to \infty$
as `cosmological time' $t = 2/\sqrt{3R} \to 0^+$, $(\EE,\tilde{g})$ is maximal in the sense
that it cannot be extended through the big bang singularity represented by the `$t=0$ part' of its causal boundary.

Consider the `standard photon' $\tilde{\lambda} \colon (0,\infty) \to \EE$ given by
$\tilde{\lambda}(s) = (0,0,3s^{1/5},s^{3/5})$. $\tilde{\lambda}(s)$ is an affinely parametrized, 
inextendible null geodesic in $\EE$ to which every inextendible null geodesics may be mapped
by an isometry. Thus $(\EE,\tilde{g})$ itself admits no complete null geodesic.

To see that $(\EE,\tilde{g})$ is conformally Einstein, let $\MM$ denote Minkowski space
and consider the mapping $\phi \colon \EE \to \MM$ given by 
$$\phi(x^1,x^2,x^3,t) = (x^1,x^2,x^3,3t^{1/3}).$$
$\phi$ is a conformal diffeomorphism of
$\EE$ onto the Minkowski upper half-space $\MM^{t>0}$ with its flat `vacuum' metric 
$g = \sum_{i=1}^3 dx^i \otimes dx^i -dt \otimes dt$, since $\phi^*(g) = (t^{-4/3}) \,\tilde{g}$.

Note that our standard
photon is mapped to $\phi \circ \tilde{\lambda}(s) = (0,0,3s^{1/5},3s^{1/5})$, which 
is inextendible on $(0,\infty)$ and has the (incomplete) affine reparametrization 
$\lambda(s) = (0,0,s,s)$. 
Thus the conclusion of the 
theorem clearly holds for this particular Einstein rescaling of $\tilde{g}$.


\medskip
\section{Proofs}\label{4}

First note that, if $g$ is an Einstein metric, $Ric(\gamma'(s),\gamma'(s)) = c \;g(\gamma'(s),\gamma'(s))$ vanishes identically along each null geodesic $\gamma$. Then from \eqref{eq:sde} we obtain $\{p,s\}=0$, which shows 
that any affine parameter along $\gamma$ is also a projective parameter. 

\begin{lem*}
If $(M,g)$ is Einstein, 
$d_M^{[g]}(x,y)$ vanishes identically along 
each complete null geodesic $\gamma(s) \colon \RR \to M$ of $g$.
\end{lem*}

\begin{proof}
To see this, assume that $s$ is an affine parameter along $\gamma$ and choose $b \in \RR$ such that $x:=\gamma(0)$ 
and $y:=\gamma(b)$ are distinct points in $M$.
Clearly for each $i \geq 2 \in \ZZ$, $\gamma_i(s) := \gamma(ibs)$ is an affine (hence, by the above remark, projective) reparametrization of $\gamma$ defined on the interval $I=(-1,1) \subset \RR$ with $\gamma_i(0)=x$ and $\gamma_i(1/i)=y$. Since the Poincar\'e distance $\rho(0,1/i) \to 0$ as $i \to \infty$, 
$\{0,1/i,\gamma_i\}_{i=2}^\infty$ is a sequence of 
Kobayashi paths joining $x$ to $y$ of arbitrarily small length. Thus $d_M^{[g]}(x,y) = 0$.

\end{proof}

\begin{proof}[Proof of the proposition.]
The first statement of the proposition is equivalent to Theorem \ref{thm:1}(e), but the proof 
is short enough to be repeated here. Assume that $(M,g)$ is a null geodescially complete Einstein 
manifold and let $x,y$ be any two distinct points in $M$. Choose a piecewise smooth null geodesic 
$\Gamma$ joining $x$ to $y$. The lemma
says that each of the finitely many smooth segments of $\Gamma$ has zero length, 
so by the triangle inequality,
$d_M^{[g]}(x,y) = 0$.

For the second statement, assume that $(M,g)$ is Einstein. If $g$ admits a single complete
null geodesic, the lemma shows that $d_M^{[g]}$ cannot be nondegenerate.

\end{proof}

\begin{proof}[Proof of the theorem.]
If the conformal metric $\tilde{g}$ satisfies the NCC and NCG,  $d_M^{[\tilde{g}]}$ is nondegenerate by Theorem \ref{thm:1}(f) above. But $d_M^{[g]} = d_M^{[\tilde{g}]}$, since $\tilde{g} \in [g]$. Applying the proposition, we see that the original Einstein metric $g$ cannot have even a single complete null geodesic. 
According to \cite{S}, null geodesic completeness is a property of the conformal class of $g$ when $M$ is 
compact, so $\tilde{g}$ must also be null geodescially incomplete.

\end{proof}


\end{document}